\documentclass[letterpaper, reqno,11pt]{article}
\usepackage[margin=1.0in]{geometry}
\usepackage{color,latexsym,amsmath,amssymb,path, amsthm,enumitem}
\usepackage[normalem]{ulem}

\newtheorem{theorem}{Theorem}[section]
\newtheorem{lemma}[theorem]{Lemma}

\newtheorem*{comPStheorem}{Theorem \ref{th:comPS}}

\theoremstyle{remark}

\newcommand{\RR}{\mathbb{R}}

\newcommand{\CC}{\mathbb{C}}
\newcommand{\BZ}{\mathbf{Z}}

\newcommand{\pts}{\mathcal P}
\newcommand{\vrts}{\mathcal V}
\newcommand{\curves}{\mathcal C}

\newcommand{\eps}{\varepsilon}

\newcommand{\smooth}{\operatorname{smooth}}
\newcommand{\sing}{\operatorname{sing}}
\newcommand{\reg}{\operatorname{reg}}

\newcommand{\ignore}[1]{}

\begin{document}
\pagenumbering{arabic}
\title{Point-curve incidences in the complex plane}

\author{
Adam Sheffer\thanks{California Institute of Technology, Pasadena, CA,
{\sl adamsh@gmail.com}.}
\and
Endre Szab\'o\thanks{Alfr\'ed R\'enyi Institute of Mathematics,
  Budapest, {\sl szabo.endre@renyi.mta.hu}.
}
\and
Joshua Zahl\thanks{University of British Columbia, Vancouver, BC,  {\sl jzahl@math.ubc.ca}.}
}
\date{\today}

\maketitle

\begin{abstract}
We prove an incidence theorem for points and curves in the complex plane. Given a set of $m$ points in $\RR^2$ and a set of $n$ curves with $k$ degrees of freedom, Pach and Sharir proved that the number of point-curve incidences is $O\big(m^{\frac{k}{2k-1}}n^{\frac{2k-2}{2k-1}}+m+n\big)$. We establish the slightly weaker bound $O_\eps\big(m^{\frac{k}{2k-1}+\eps}n^{\frac{2k-2}{2k-1}}+m+n\big)$ on the number of incidences between $m$ points and $n$ (complex) algebraic curves in $\CC^2$ with $k$ degrees of freedom. We combine tools from algebraic geometry and differential geometry to prove a key technical lemma that controls the number of complex curves that can be contained inside a real hypersurface. This lemma may be of independent interest to other researchers proving incidence theorems over $\CC$.
\end{abstract}
\maketitle
\section{Introduction}
Let $\pts$ be a set of points and let $\vrts$ be a set of geometric objects (for example, one might consider lines, circles, or planes) in a vector space $K^d$ over a field $K$. An \emph{incidence} is a pair $(p,V)\in \pts \times \vrts$ such that the point $p$ is contained in the object $V$. In incidence problems, one is usually interested in the maximum number of incidences in $\pts \times \vrts$, taken over all possible sets $\pts,\vrts$ of a given size. For example, the well-known Szemer\'edi-Trotter Theorem \cite{Szemeredi} states that any set of $m$ points and  $n$ lines in $\RR^2$ must have $O(m^{2/3}n^{2/3}+m+n)$ incidences.

Incidence theorems have a large variety of applications. For example, in the last few years they have been used by Guth and Katz \cite{Guth} to almost completely settle Erd\H os' distinct distances problem in the plane; by Bourgain and Demeter \cite{BD13,BD14} to study restriction problems in harmonic analysis; by Raz, Sharir, and Solymosi \cite{RSS14} to study expanding polynomials; and by Farber, Ray, and Smorodinsky \cite{FRS13} to study properties of totally positive matrices.

\subsection{Previous work}
We will be concerned with the number of incidences between points and various classes of curves. Later, we will define several different types of curves, but for the definition below one can think of a curve as merely a subset of $K^2$, where $K$ is either $\RR$ or $\CC$.

Let $\mathcal{C}$ be a set of curves in $K^2$ and let $\pts$ be a set of points in $K^2$. We say that the arrangement $(\pts,\mathcal{C})$ has $k$ \emph{degrees of freedom} and \emph{multiplicity type} $s$ if
\begin{itemize}
 \item For any subset $\pts'\subset\pts$ of size $k$, there are at most $s$ curves from $\curves$ that contain $\pts'$.
 \item Any pair of curves from $\curves$ intersect in at most $s$ points from $\pts$.
\end{itemize}

We will use $I(\pts,\curves)$ to denote the number of incidences between the points in $\pts$ and curves in $\curves$.
The current best bound for incidences between points and general curves in $\RR^2$ is the following (better bounds are known for some specific types of curves, such as circles and parabolas\footnote{Recently, Sharir and the third author obtained an improvement \cite{SZ} to Theorem \ref{th:PS} whenever the curves are algebraic.}).

\begin{theorem}[Pach and Sharir \cite{PS}] \label{th:PS}
Let $\pts$ be a set of $m$ points in $\RR^2$ and let $\curves$ be a set of $n$ simple plane curves. Suppose that $(\pts,\curves)$ has $k$ degrees of freedom and multiplicity type $s$. Then
\begin{equation*}
I(\pts,\curves) = O_{k,s}\big(m^{\frac{k}{2k-1}}n^{\frac{2k-2}{2k-1}}+m+n\big).
\end{equation*}
\end{theorem}

If the curves are algebraic, then we can drop the requirement that the curves are simple (however, the implicit constant will now depend on the degree of the curves). This special case was proved several years earlier than Theorem \ref{th:PS}. The proof follows from the techniques in \cite{Clarkson}, and appears explicitly in \cite{PS92}.
\begin{theorem}[Pach and Sharir \cite{PS92,PS}] \label{th:PSAlg}
Let $\pts$ be a set of $m$ points in $\RR^2$ and let $\curves$ be a set of $n$ algebraic curves of degree at most $D$.  Suppose that $(\pts,\curves)$ has $k$ degrees of freedom and multiplicity type $s$. Then
\begin{equation} \label{PSCTheoremEqn}
I(\pts,\curves) = O_{k,s,D}\big(m^{\frac{k}{2k-1}}n^{\frac{2k-2}{2k-1}}+m+n\big).
\end{equation}

\end{theorem}

Less is known about point-curve incidences in the complex plane. If we
add the additional requirement that pairs of curves must intersect
transversely\footnote{That is, whenever two complex curves intersect
  at a smooth point of both curves, their complex tangent lines at the
  point of intersection are distinct.}, then an analogue of Theorem
\ref{th:PSAlg} can be proved using the techniques of Solymosi-Tao from
\cite{ST}, although these methods introduce an $\eps$ loss in the
exponent. Previously, T\'oth \cite{Toth14} proved the important
special case where the curves in $\curves$ are lines. This was
generalized by the third author in \cite{zahl}, who proved a bound
analogous to that in Theorem \ref{th:PSAlg} for complex
curves. However, in addition to the requirement that curves intersect
transversely, the results of \cite{zahl} have an additional
restriction on the relative sizes of $\pts$ and $\curves$, and they
require that the curves be smooth.
  Elekes and the second author \cite[Theorem 9]{elekes2012find}
  proved Pach-Sharir-like estimates for arbitrary complex subvarieties in $\CC^d$,
  but their exponent is far from optimal.
Finally, Dvir and Gopi \cite{DG15}
and the third
author \cite{zahl2} considered incidences between points
and lines in $\CC^d$, for any $d\ge 3$.

Asking for the curves to intersect transversely is rather restrictive; some of the simplest cases such as incidences with circles or parabolas do not satisfy this requirement. If we do not require that pairs of curves intersect transversely, then much less is known. Very recently, Solymosi and de Zeeuw \cite{SdZ14} proved a complex analog of Theorem \ref{th:PSAlg}, but only for the special case where the point set is a Cartesian product $A\times B \subset \CC$. This bound has already been used to prove several results---see \cite{RSZ14,VdZ14}.

\subsection{New results}
We obtain a complex analogue of Theorem \ref{th:PSAlg}, although our version introduces an $\eps$ loss in the exponent.

\begin{theorem} \label{th:comPS}
For each $k\geq 1, D\geq 1,\ s\geq 1,$ and $\epsilon>0$, there is a constant $C=C_{\epsilon,D,s,k}$ so that the following holds. Let $\pts\subset\CC^2$ be a set of $m$ points and let $\curves$ be a set of $n$ complex algebraic curves of degree at most $D$. Suppose that $(\pts,\curves)$ has $k$ degrees of freedom and multiplicity type $s$. Then
\begin{equation}\label{comPSThmEqn}
I(\pts,\curves) \le C\big(m^{\frac{k}{2k-1}+\eps}n^{\frac{2k-2}{2k-1}}+m+n\big).
\end{equation}
\end{theorem}

The new improvement is that Theorem \ref{th:comPS} does not require the curves to intersect transversely.
The main new tool in the proof is the Picard--Lindel\"of theorem.

\ignore{ 
\subsection{Proof sketch}
Each point of $\pts$ can be regarded as a point in $\RR^4$, and each curve of $\curves$ can be regarded as a two-dimensional variety in $\RR^4$. Thus the problem is reduced to bounding the number of incidences between a set of points in $\RR^4$ and a set $\mathcal{S}$ of two-dimensional surfaces in $\RR^4$. If every pair (or at least most pairs) of surfaces in $\mathcal{S}$ intersect transversely, then the bound \eqref{comPSThmEqn} can be obtained by using the techniques of Solymosi and Tao from \cite{ST}. However, if many pairs of surfaces in $\mathcal{S}$ fail to intersect transversely, then the techniques from \cite{ST} do not apply.

Luckily, the surfaces in $\mathcal{S}$ are special---they come from complex curves in $\CC^2$. More precisely, the surfaces in $\mathcal{S}$ are the images of complex curves in $\CC^2$ under the usual embedding of $\CC^2$ into $\RR^4$.
As we will show below, this means that the only way that many surfaces in $\mathcal{S}$ can lie in a common low-degree hypersurface $Z$ is if they intersect only in singular points of $Z$.
} 

\section{Preliminaries}\label{algGeoToolsSec}
\subsection{Varieties and ideals} \label{ssec:Quant}

In this paper we work over the fields $\RR$ and $\CC$.
Let $K=\RR$ or $\CC$.
Varieties are (possibly reducible) Zariski closed subsets of $K^d$.
If $X\subset K^d$ is a set, let $\overline{X}$
be the Zariski closure of $X$; this is the smallest variety in $K^d$
that contains $X$.

If $Z\subset\RR^d$ is a variety, let $Z^*\subset\CC^d$ be the smallest complex variety containing $Z$; i.e., $Z^*$ is obtained by embedding $Z$ into $\CC^d$ and then taking the Zariski closure.
If $Z\subset\CC^d$, let $Z(\RR)\subset\RR^d$ be the set of real points of $Z$.
We also identify $\CC^2$ with $\RR^4$ using the map $\iota(x_1+iy_1,x_2+iy_2)=(x_1,y_1,x_2,y_2)$ (where $x_1,y_1,x_2,y_2\in \RR$).
If $\mathcal{C}$ is a set of curves in $\CC^2$, we define $\iota(\mathcal{C})=\{\iota(\gamma)\colon\gamma\in\mathcal{C}\}$.

If $Z\subset K^d$ is a variety, let $I(Z)$ be the ideal of polynomials in $K[x_1,\ldots,x_d]$ that vanish on $Z$. If $I\subset K[x_1,\ldots,x_d]$ is an ideal, let $\BZ(I)\subset K^d$ be the intersection of the zero-sets of all polynomials in $I$. Sometimes it will be ambiguous whether an ideal is a subset of $\RR[x_1,\ldots,x_d]$ or $\CC[x_1,\ldots,x_d]$. To help resolve this ambiguity, we will write $\BZ_{\RR}(I)$ or $\BZ_{\CC}(I)$. If $P\in K[x_1,\ldots,x_d]$ is a polynomial, we abuse notation and write $\BZ(P)$ instead of $\BZ((P))$. If $I\subset\CC[x_1,\ldots,x_d]$ is an ideal, we use $\sqrt{I}=I(\BZ(I))$ to denote the radical of $I$.

Often in our arguments we will refer to properties that hold for most points on a variety. To make this precise, we will introduce the notion of a generic point.
Let $Z\subset\CC^d$ be an irreducible variety, and let $M$ be a finite set of polynomials, none of which vanish on $Z$. We say that a point $z\in Z$ is \emph{generic} (with respect to $M$) if none of the polynomials in $M$ vanish at $z$. In particular, for $Z$ and $M$ fixed, the set of generic points of $Z$ is Zariski dense in $Z$.

In practice, the set of polynomials will be apparent from context, so we will abuse notation and simply refer to generic points. In general, the set of polynomials $M$ will depend on the variety $Z$, the points and curves from the statement of Theorem \ref{th:comPS}, any previously defined objects, and whatever property is currently under consideration.

If $Z(\RR)$ is Zariski dense in $Z$, then we define a generic real point of $Z(\RR)$ to be a point $z\in Z(\RR)$ for which no polynomial in $M$ vanishes. In particular, if $Z(\RR)$ is dense in $Z$, then $Z$ always contains a generic real point.

Finally, we will sometimes refer to generic linear spaces or generic linear transformations. A generic linear space of dimension $\ell$ in $\CC^d$ is a generic point of the Grassmannian of $\ell$-dimensional vector spaces in $\CC^d$. Similarly, a generic linear transformation in $\CC^d$ is a generic point of $\operatorname{GL}(\CC,d)$.

The degree of an irreducible affine variety $V \subset\CC^d$ of dimension $d^\prime$ is the number of points of the intersection of $V$ with a generic linear space of dimension $d-d^\prime$ (for several equivalent definitions, see \cite[Chapter 18]{Harris}).
We define the degree of a reducible variety $V$ as the sum of the degrees of the irreducible components of $V$ (note that these components may have different dimension). In practice, we are only interested in showing that the degrees of various varieties are bounded, so the specific definition of degree is not too important.

\begin{lemma}[Varieties and their defining ideals]\label{varDefIdeal}
Let $Z\subset\CC^d$ be a variety of degree $C$.
Then there exist polynomials $f_1,\ldots,f_{\ell}$ such that $(f_1,\ldots,f_{\ell})=I(X)$ and $\sum_{j=1}^{\ell} \deg f_j = O_{C,d}(1)$.
\end{lemma}
\begin{proof}
This is essentially \cite[Theorem A.3]{BGT}. In \cite{BGT}, the authors prove the weaker statement that there exists a set of polynomials $g_1,\ldots,g_t$ such that $\sum\deg g_j = O_{d,C}(1)$ and $I(Z) = \sqrt{(g_1,\ldots,g_t)}$. However, a set of generators for $\sqrt{(g_1,\ldots,g_t)}$ can then be computed using Gr\"obner bases (see e.g.~\cite{CLS} for an introduction to Gr\"obner bases). The key result is due to Dub\'e \cite{dube}, which says that a reduced Gr\"obner basis for $(g_1,\ldots,g_t)$ can be found (for any monomial ordering) such that the sum of the degrees of the polynomials in the basis is $O_{d,C^\prime}(1),$ where $C^\prime= \sum\deg g_j$. Since $C^\prime=O_{d,C}(1)$, we conclude that the sum of the degrees of the polynomials in the Gr\"obner basis is $O_{d,C}(1)$. Once a Gr\"obner basis for $(g_1,\ldots,g_t)$ has been obtained, a set of generators for $\sqrt{(g_1,\ldots,g_t)}$ can then be computed (see e.g.~\cite[Section 9]{GTZ}).
\end{proof}

\subsection{Regular points, singular points, and smooth points}
We will often refer to the \emph{dimension} of an affine real algebraic variety. Informally, a real algebraic variety $X$ has dimension $d^\prime$ if there exists a subset of $X$ that is homeomorphic to the open $d^\prime$-dimensional cube, but there does not exist a subset of $X$ that is homeomorphic to the open $(d^\prime+1)$-dimensional cube. See \cite{BCR} for a precise definition of the dimension of a real algebraic variety.

Let $X\subset\RR^d$ be a variety of dimension $d^\prime$ and let
$\zeta\in X$. We say that $\zeta$ is a \emph{smooth} point of $X$ if
there is a Euclidean neighborhood $U\subset\RR^d$ containing $\zeta$
such that $X\cap U$ is a $d^\prime$-dimensional
embedded submanifold;
for example, see \cite[Section 3.3]{BCR}. In this paper we only consider smooth manifolds, and for brevity we refer to these simply as manifolds. Let $X_{\smooth}$ be the set of smooth points of $X$; then $X_{\smooth}$ is a $d^\prime$-dimensional smooth manifold.

Similarly, let $X\subset\CC^d$ be a variety of dimension $d^\prime$
and let $\zeta\in X$. We say that $\zeta$ is a \emph{smooth} point of
$X$ if there is a Euclidean neighborhood $U\subset\CC^{d}$ containing
$\zeta$ such that $X\cap U$ is a $d^\prime$-dimensional
embedded complex submanifold.
Again, let $X_{\smooth}$ be the set of smooth point of $X$; then $X_{\smooth}$ is a $d^\prime$-dimensional complex manifold.

Let $X\subset\CC^d$ be a variety of pure dimension $d^\prime$ (i.e., all irreducible components of $X$ have dimension $d^\prime$), and let $f_1,\ldots,f_\ell$ be polynomials that generate $I(X)$. We say that $\zeta\in X$ is a \emph{regular} point of $X$ if
\begin{equation} \label{JacobiMatrix}
\operatorname{rank}\left[\begin{array}{c}\nabla f_1(\zeta)\\ \vdots \\ \nabla f_\ell(\zeta)\end{array}\right]=d-d^\prime.
\end{equation}
Let $X_{\reg}$ be the set of regular points of $X$. If $\zeta\in X$ is not a regular point of $X$, then $\zeta$ is a \emph{singular} point of $X$. Let $X_{\sing}$ be the set of singular points of $X$.

\begin{lemma}[\cite{mumford}, Corollary 1.26]\label{equivalenceOfSmoothSing}
Let $X\subset\CC^d$ be a variety of pure dimension $d^\prime$. Then $X_{\smooth}=X_{\reg}$.
\end{lemma}

\begin{lemma} \label{le:singular}
Let $X\subset\CC^d$ be a variety of degree $C$. Then $X_{\sing}$ is a variety of dimension strictly smaller than $\dim(X)$, and $\deg(X_{\sing})=O_{C,d}(1)$.
\end{lemma}
\begin{proof}
By Lemma \ref{varDefIdeal}, there exist polynomials $f_1,\ldots,f_{\ell}$ such that $(f_1,\ldots,f_{\ell})=I(X)$ and $\sum_{j=1}^{\ell} \deg f_j = O_{C,d}(1)$. We have
\begin{equation}\label{expressionFOrZSing}
X_{\sing}= \bigg\{\zeta\in X\colon \operatorname{rank}\left[\begin{array}{c}\nabla f_1(\zeta)\\ \vdots \\ \nabla f_\ell(\zeta)\end{array}\right]<d-d^\prime\bigg\}.
\end{equation}
Equation \eqref{expressionFOrZSing} shows that $X_{\sing}$ can be written as the zero locus of $O_{\ell}(1)=O_{d,C}(1)$ polynomials, each of degree $O_{d,C}(1)$. Thus $X_{\sing}$ is a variety of degree $O_{d,C}(1)$. It remains to prove that $X_{\sing}$ has dimension strictly smaller than $\dim(X)$. This property can be found, for example, in \cite[Chapter I, Theorem 5.3]{hartshorne}.
\end{proof}

\section{Images of complex curves in a real variety} \label{sec:TangentBundle}

The goal of this section is to prove Lemma \ref{HyperSurfaceOneLeafThm}.
In this lemma we consider a hypersurface $V \subset \RR^4$, and study the behavior of complex curves $\gamma \subset \CC^2$ that satisfy $\iota(\gamma) \subset V$.

Let $M\subseteq\RR^n$ be a submanifold and $\zeta\in M$ a point.
We identify the tangent space $T_\zeta\RR^n$ with $\RR^n$ itself,
so $T_\zeta M$ becomes a linear subspace of $\RR^n$.
Analogously, let $N\subseteq\CC^n$ be a variety and
$x\in N_{\reg}$ a smooth point.
Then we identify $T_x\CC^n$ with the complex vector space $\CC^n$,
and $T_xN$ becomes a complex linear subspace of $\CC^n$.

Let $M$ be a $d$-dimensional smooth manifold. The \emph{tangent bundle} $TM$ is a $2d$-dimensional smooth manifold that is the disjoint union of the tangent spaces $\{T_\zeta\}_{\zeta\in M}$. Each element of the tangent bundle can be identified with a pair $(\zeta,v)$, where $\zeta\in M$ and $v\in T_{\zeta}M$.

Let $E\subset TM$ be a $(d+d^\prime)$-dimensional sub-manifold of $TM$. We say that $E$ is a $d^\prime$-dimensional \emph{sub-bundle} of $TM$ if for every $\zeta\in M$, we have $(\{\zeta\}\times T_\zeta M)\cap E=\{\zeta\}\times V$, where $V$ is a $d^\prime$-dimensional vector subspace of $ T_\zeta M=\RR^d$. We will call this subspace $E(\zeta)\subset\RR^d$. Intuitively, the vector space $E(\zeta)$ varies smoothly as the base-point $\zeta$ changes.

A \emph{vector field} on $M$ is a smooth function $X\colon M\to TM$
that assigns an element of $T_\zeta M$ to each point $\zeta\in M$. We
will abuse notation slightly and write $X(\zeta)=v$ to mean
$X(\zeta)=(\zeta,v)\in TM$. If $E$ is a sub-bundle of $TM$ and
$X\colon M\to TM$ is a vector field, we say that $X$ \emph{takes
  values in} $E$ if $X(\zeta)\in E$ for all $\zeta\in M$.


The following is a variant of the Picard--Lindel\"of theorem (e.g., see \cite{KP10}).

\begin{theorem}\label{PicardLindelof}
Let $X$ be a smooth vector field on a manifold $M$ and $\zeta\in M$ a point where $X(\zeta) \ne 0$.
Then for any sufficiently small $\varepsilon>0$ there exists a unique smooth arc $\alpha:[-\varepsilon,\varepsilon]\to M$ starting at $\zeta$ whose tangent vectors are in $X$; that is, a unique arc $\alpha$ that solves the initial
  value problem
  \begin{equation}
    \label{eq:1}
    \alpha(0)=\zeta,
    \quad\quad
    \dot\alpha(t) = X\big(\alpha(t)\big)
    \quad
    \text{for all }t\in[-\varepsilon,\varepsilon].
  \end{equation}
\end{theorem}

We are now ready to show that if $Z$ is a bounded-degree hypersurface in $\RR^4$, then for a generic point $z\in Z$ there is at most one irreducible curve $\gamma\subset\CC^2$ that satisfies $z\in \iota(\gamma) \subset Z$.

\begin{lemma}\label{HyperSurfaceOneLeafThm}
Let $P\in\RR[x_1,y_1,x_2,y_2]$ be a polynomial of degree at most $D$.
Then for every
$p\in \BZ_{\RR}(P)\backslash Z_{\CC}(P)_{\sing}$,
there is at most one irreducible complex curve
$\gamma\subset\CC^2$ with $p\in\iota(\gamma_{\reg})$ and
$\iota(\gamma)\subset \BZ_{\RR}(P).$
\end{lemma}
\begin{proof}
We set $M = \BZ_\RR(P)\setminus Z_{\CC}(P)_{\sing}$
and note that $M$, if non-empty,
is a three-dimensional submanifold in $\RR^4$.
The isomorphism $\iota$ carries the multiplication by $i$ in
$\CC^2$ into the linear transformation
\[ J:\RR^4\to\RR^4,\quad\quad J(x_1,y_1,x_2,y_2) = (-y_1,x_1,-y_2,x_2). \]

Notice that for any vector $v\in\RR^4$ we have $J(J(v)) = -v$.
Thus, for any linear subspace $V\subset \RR^4$ we have $J(J(V)) = V$.
Since $J$ corresponds to multiplication by $i$ in $\CC^2$, a linear subspace $V$ is $J$-invariant if and only if $V=\iota(V')$ for some complex subspace $V'\le\CC^2$.
In particular, all $J$-invariant subspaces are even dimensional.

For every point $p\in M$ we define the linear subspace $E_p = T_p M \cap J^{-1}(T_p M)$.
Intuitively, $E_p$ is the largest subset of $T_p M$ that is invariant under $J$.
Since the linear subspace $T_p M$ is three-dimensional, it cannot be $J$-invariant.
This implies that $J^{-1}(T_p M)$ is a different three-dimensional subspace,
and thus $E_p$ is a two-dimensional linear subspace.
As $p$ varies, the union of the $p \times E_p$ forms a two-dimensional
sub-bundle $E$ of the tangent bundle $TM$.

Fix a point $p\in M$, and choose a vector field $X$ defined
in an open neighbourhood $U\subseteq M$ of $p$
which takes values in $E$, and $X(p)\ne0$.
By Theorem \ref{PicardLindelof} there is a unique arc
$\alpha:[-\varepsilon,\varepsilon]\to\gamma_{\reg}$
that solves \eqref{eq:1}.

Consider an irreducible complex curve $\gamma\subset\CC^2$
that satisfies
$\iota(\gamma) \subset \BZ_{\RR}(P)$ and
$p\in \iota(\gamma_{\reg})$.
For any point $q\in\iota(\gamma_{\reg})$
the tangent space $T_{\iota^{-1}(q)}\gamma$
is a complex line in $\CC^2$,
hence
$T_{q}\iota(\gamma)=\iota\big(T_{\iota^{-1}(q)}\gamma)$
is a $J$-invariant 2-plane in $\RR^4$
which is contained in $T_{q}M$.
This implies that $T_q\iota(\gamma)=E_q$,
so $X(q)$ is tangent to $\iota(\gamma)$.
By applying Theorem \ref{PicardLindelof} to the manifold
$\iota(\gamma_{\reg})$, and the restriction of $X$ to $\iota(\gamma_{\reg})$,
we obtain an arc
$\beta:[-\varepsilon',\varepsilon']\to\iota(\gamma_{\reg})$
that solves the same equation \eqref{eq:1}.
While we might get that $\varepsilon' \neq \varepsilon$,
the uniqueness of the solution implies that $\alpha$ and $\beta$
must have a common sub-arc $\alpha'$ around $p$.
Since $\alpha'$ is an infinite set,
the complex curve $\gamma$ must be the Zariski closure of $\iota^{-1}(\alpha')$.

Suppose now that $\alpha''\subset\alpha$ is any sub-arc
around $p$ such that the Zariski closure of $\iota^{-1}(\alpha'')$
is an irreducible curve $\gamma''\subset\CC^2$.
By the above argument $\gamma$ is the Zariski closure of
$\iota^{-1}(\alpha'\cap\alpha'')$,
hence $\gamma=\gamma''$.
This proves that $\gamma$, if exists, is uniquely determinded by $\alpha$.
\end{proof}

\section{Proof of Theorem \ref{th:comPS}}\label{proofOfThmSection}
We are now ready to prove Theorem \ref{th:comPS}. For the reader's convenience we will restate it here.
\begin{comPStheorem}
For each $k\geq 1,\ D\geq 1,\ s\geq 1,$ and $\epsilon>0$, there is a constant $C=C_{\epsilon,D,s,k}$ such that the following holds. Let $\pts\subset\CC^2$ be a set of $m$ points and let $\curves$ be a set of $n$ complex algebraic curves of degree at most $D$. Suppose that $(\pts,\curves)$ has $k$ degrees of freedom and multiplicity type $s$. Then
\begin{equation*}
I(\pts,\curves) \le C\big(m^{\frac{k}{2k-1}+\eps}n^{\frac{2k-2}{2k-1}}+m+n\big).
\end{equation*}
\end{comPStheorem}
\begin{proof}
We will make crucial use of the Guth-Katz polynomial partitioning technique from \cite[Theorem 4.1]{Guth}.
\begin{theorem}\label{th:partition}
Let $\pts$ be a set of $m$ points in $\RR^d$. For each $r\geq 1$, there exists a polynomial $P$ of degree at most $r$ such that $\RR^{d}\backslash\BZ(P)$ is a union of $O(r^d)$ connected components (cells), and each cell contains $O(m/r^d)$ points of $\pts$.
\end{theorem}

Since the curves of $\curves$ have $k$ degrees of freedom, the K\H ov\' ari-S\'os-Tur\'an theorem (e.g., see
\cite[Section 4.5]{Mat02}) implies $I(\pts,\curves) = O(mn^{1-1/k} + n)$. When $m=O(n^{1/k})$, this implies the bound $I(\pts,\curves) = O(n)$. Thus, we may assume that
\begin{equation} \label{eq:kst}
n=O\left(m^k\right).
\end{equation}

We will prove by induction on $m+n$ that
\begin{equation*}
I(\pts,\curves) \le \alpha_{1} m^{\frac{k}{2k-1}+\varepsilon}n^{\frac{2k-2}{2k-1}}+\alpha_{2}(m+n),
\end{equation*}
where $\alpha_{1},\alpha_{2}$ are sufficiently large constants. The base case where $m+n$ is small can be handled by choosing sufficiently large values of $\alpha_{1}$ and $\alpha_{2}$. In practice, we will bound $I(\iota(\pts),\iota(\curves))$. Since $\iota\colon\CC^2\to\RR^4$ is a bijection, $I(\pts,\curves)=I(\iota(\pts),\iota(\curves))$.

\paragraph{Partitioning $\RR^4$.}
Let $P$ be a partitioning polynomial of degree at most $r$, as described in Theorem \ref{th:partition}.
The constant $r$ is taken to be sufficiently large, as described below.
The asymptotic relations between the various constants in the proof are
\begin{equation*}
 2^{1/\varepsilon} \ll r \ll \alpha_{2} \ll \alpha_{1}.
\end{equation*}

Let $\Omega_1, \ldots, \Omega_{\ell}$ be the cells of the partition; we have $\ell=O(r^4)$. Let $\vrts_i$ be the set of varieties from $\iota(\curves)$ that intersect the interior of $\Omega_i$ and let $\pts_i$ be the set of points $p\in\pts$ such that $\iota(p)\in \Omega_i$. Let $m_j=|\pts_j|$, $m' = \sum_{j=1}^{\ell} m_j$, and $n_j=|\vrts_j|$. By Theorem \ref{th:partition}, $m_j=O( m/r^4)$ for every $1\le j \le \ell$.

By \cite[Theorem A.2]{ST}, every variety from $\vrts$ intersects $O(r^2)$ cells of $\RR^4\setminus \BZ(P)$.
Therefore, $\sum_{j=1}^{\ell} n_j = O\left(nr^2\right)$. Combining this with H\"older's inequality implies
\begin{align*}
\sum_{j=1}^{\ell} n_j^{\frac{2k-2}{2k-1}} &
=O\left(\left(nr^{2}\right)^{\frac{2k-2}{2k-1}}\ell^{\frac{1}{2k-1}}\right) =O\left(n^{\frac{2k-2}{2k-1}}r^{\frac{4k}{2k-1}}\right).
\end{align*}

By the induction hypothesis, we have
\begin{align*}
\sum_{j=1}^{\ell} I(\pts_j,\vrts_j) &\le \sum_{j=1}^{\ell} \left(\alpha_{1} m_j^{\frac{k}{2k-1}+\eps}n_j^{\frac{2k-2}{2k-1}}+\alpha_{2}(m_j+n_j)\right) \\
&\le O\left(\alpha_{1} m^{\frac{k}{2k-1}+\eps}r^{-\frac{4k}{2k-1}-4\eps} \sum_{j=1}^{\ell} n_j^{\frac{2k-2}{2k-1}} \right) + \sum_{j=1}^{\ell}\alpha_{2}(m_j+n_j) \\[2mm]
&\le O\left(\alpha_{1} r^{-\varepsilon}m^{\frac{k}{2k-1}+\eps}n^{\frac{2k-2}{2k-1}}   \right) + \alpha_{2}\left(m'+O\left(nr^2\right)\right).
\end{align*}

By \eqref{eq:kst}, we have $n^{\frac{1}{2k-1}}=O\left(m^{\frac{k}{2k-1}}\right)$, which in turn implies $n=O\left(m^{\frac{k}{2k-1}}n^{\frac{2k-2}{2k-1}}\right)$.
Thus, when $\alpha_{1}$ is sufficiently large with respect to $r$ and $\alpha_{2}$, we have
\begin{equation*}
\sum_{j=1}^{\ell} I(\pts_j,\vrts_j) = O\left(\alpha_{1} r^{-\varepsilon}m^{\frac{k}{2k-1}+\varepsilon}n^{\frac{k}{2k-1}} \right) + \alpha_{2}m'.
\end{equation*}

By taking $r$ to be sufficiently large with respect to $\eps$ and the implicit constant in the $O$-notation, we have
\begin{equation*}
\sum_{j=1}^{\ell} I(\pts_j,\vrts_j) \le \frac{\alpha_{1}}{2} m^{\frac{k}{2k-1}+\eps}n^{\frac{2k-2}{2k-1}} + \alpha_{2}m',
\end{equation*}
i.e.,
\begin{equation}\label{eq:incCells}
I(\iota(\pts)\backslash\BZ_{\RR}(P),\iota(\curves))\le \frac{\alpha_{1}}{2} m^{\frac{k}{2k-1}+\eps}n^{\frac{2k-2}{2k-1}} + \alpha_{2}m'.
\end{equation}

\paragraph{Incidences on the partitioning hypersurface.}
It remains to bound incidences with points that are on the partitioning hypersurface $\BZ(P)$. To do this, we will make use of the point-curve bound from Theorem \ref{th:PSAlg}.
\begin{lemma}\label{1DcurveIncidenceBd}
Let $\pts\subset\CC^2$. Let $\curves$ be a set of complex curves of degree at most $C_0$ such that $(\pts,\curves)$ has $k$ degrees of freedom and multiplicity type $s$. Let $Y\subset\CC^4$ be an algebraic variety of degree at most $C_1$. Suppose that for each $\gamma\in\curves$, the intersection $\iota(\gamma)\cap Y(\RR)$ is a real algebraic variety of dimension at most one. Then
\begin{equation}\label{boundOnCurveIncidences}
I(\iota(\pts)\cap Y(\RR),\ \iota(\curves))= O(|\pts|^{k/(2k-1)}|\curves|^{(2k-2)/(2k-1)}+|\pts|+|\curves|),
\end{equation}
where the implicit constant depends on $k$, $s$, $C_0$, and $C_1$.
\end{lemma}
\begin{proof}
Let $\pi\colon\RR^4\to\RR^2$ be a generic linear transformation (see Section \ref{ssec:Quant}). Then for each $\gamma\in\curves,$ $\pi(\iota(\gamma)\cap Y(\RR))\subset\RR^2$ is the zero set of a non-zero polynomial of degree $O_{C_0,C_1}(1)$ (e.g., see \cite[Section 5.1]{ST}); each set of this form is a union of plane curves and a finite set of points.

Let $\Gamma=\{\pi\big(\iota(\gamma)\cap Y(\RR)\big)\colon \gamma\in\curves\}$. Then $\Gamma$ is a finite set of (not necessarily irreducible) plane algebraic curves and isolated points, and $\big(\pi(\iota(\pts)),\Gamma\big)$ has $k$ degrees of freedom and multiplicity type $O_{s,C_0,C_1}(1)$.

By Theorem \ref{th:PSAlg},
\begin{equation}\label{planarCurveBound}
I(\pi(\iota(\pts)),\Gamma)=O(|\pts|^{k/(2k-1)}|\curves|^{(2k-2)/(2k-1)}+|\pts|+|\curves|),
\end{equation}
where the implicit constant depends on $k,$ $s$, $C_0,$ and $C_1$. Since each incidence in $I(\iota(\pts)\cap Y(\RR),\ \iota(\curves))$ appears as an incidence in $I(\pi(\iota(\pts)),\Gamma)$, \eqref{planarCurveBound} implies \eqref{boundOnCurveIncidences}.
\end{proof}

We are now ready to bound the number of incidences involving points lying on $\BZ_{\RR}(P)$. Let $\pts_0= \iota(\pts) \cap \BZ_{\RR}(P)$, let $m_0=|\pts_0|=m-m'$, and let $\curves_0=\{\gamma\in\curves\colon\iota(\gamma)\subset\BZ_{\RR}(P)\}$.
By Lemma \ref{le:singular}, for each $\gamma\in\curves$, we have that
  $\iota(\gamma)_{\sing}=\iota(\gamma_{\sing})$ is a finite set of size
  $O_D(1)$, hence
  \begin{equation}\label{incidencesSing}
    |\{(p,\gamma)\in \pts_0 \times \curves :\,
    \iota(p)\in\iota(\gamma)_{\sing}\}|=O_D(n).
  \end{equation}

Let $Y$ be the real part of $Z_{\CC}(P)_{\sing}$.
We apply Lemma \ref{HyperSurfaceOneLeafThm} to $P$, to obtain
\begin{equation}\label{incidencesOnSmoothLeafs}
|\{(p,\gamma)\in\pts_0\times \curves_0 :\, \iota(p)\in \BZ_{\RR}(P)\backslash Y,\ \iota(p)\in \iota(\gamma)_{\reg}\}|\leq m_0.
\end{equation}

 Lemma \ref{le:singular} implies that $Z_{\CC}(P)_{\sing}$ is a variety of degree $O_{r}(1)$ and dimension at most two.
 This in turn implies that $O_{r}(1)$ varieties of the form $\iota(\gamma)$ are contained in $Y$.
 Thus
\begin{equation}\label{incidencesSurfaceInY}
|\{(p,\gamma)\in\pts_0\times \curves_0 :\, \iota(p)\in \iota(\gamma)_{\reg},\ \iota(\gamma)\subset Y\}|=O_{r}(m_0).
\end{equation}

Let $\curves' = \curves \setminus \curves_0$.
It remains to control the size of the sets
\begin{equation*}
\{(p,\gamma)\in\pts_0\times \curves' :\, \iota(p)\in \iota(\gamma)_{\reg} \}
\end{equation*}
and
\begin{equation*}
\{(p,\gamma)\in\pts_0\times \curves_0 :\, \iota(\gamma)\not\subset Y, \, \iota(p)\in (\iota(\gamma)_{\reg} \cap Y) \}.
\end{equation*}

By Lemma \ref{1DcurveIncidenceBd}, both of these sets have size
\begin{equation}\label{incidencesSurfaceNotInZiOrY}
O(m_0^{k/(2k-1)}n^{(2k-2)/(2k-1)}+m_0+n).
\end{equation}

Combining \eqref{incidencesSing}, \eqref{incidencesOnSmoothLeafs}, \eqref{incidencesSurfaceInY}, and \eqref{incidencesSurfaceNotInZiOrY} implies
\begin{align*}
I(\pts_0,\iota(\curves)) =  O(m_0^{k/(2k-1)}n^{(2k-2)/(2k-1)}+m_0+n).
\end{align*}

Taking $\alpha_1,\alpha_2$ to be sufficiently large with respect to the constant of the $O$-notation, we have
\begin{equation}\label{boundPts0Vrts}
I(\iota(\pts)\cap\BZ_{\RR}(P), \iota(\curves)) \le \frac{\alpha_1}{2}m^{k/(2k-1)}n^{(2k-2)/(2k-1)}+\alpha_2(m_0+n).
\end{equation}
Combining \eqref{boundPts0Vrts} and \eqref{eq:incCells} completes the induction.
\end{proof}

\paragraph{Acknowledgements.}
The authors would like to thank Orit Raz and Frank de Zeeuw for a
discussion that pushed us to work on this problem, and L\'aszl\'o Lempert
for finding an error in an earlier version of the proof.
We would like to thank the anonymous referee for numerous suggestions and recommendations. Part of this research was performed while the authors were visiting the Institute for Pure and Applied Mathematics (IPAM) in Los Angeles, which is supported by the National Science Foundation. The second author was supported by National Research, Development and Innovation Office (NKFIH) Grants K115799, K120697, ERC\_HU\_15 118286. 
The third author was supported in part by an NSF Postdoctoral Fellowship.

\bibliographystyle{abbrv}
\bibliography{complexPSBiblio}
\end{document}